\documentclass{amsart}
\usepackage[foot]{amsaddr}
\usepackage{amsmath}
\usepackage{amssymb}
\newtheorem{Th}{Theorem} 
\newtheorem{Prop}{Proposition}
\newtheorem{Conjecture}{Conjecture}
\theoremstyle{definition}
\newtheorem{definition}{Definition}
\newtheorem{Remark}{Remark} 

\usepackage{color}
\usepackage[normalem]{ulem}
\definecolor{DarkGreen}{rgb}{0,0.5,0.1} 

\newcommand\soutD{\bgroup\markoverwith
{\textcolor{DarkGreen}{\rule[.5ex]{2pt}{1pt}}}\ULon}

\begin{document}
\title[Spectral isoperimetric inequalities]{Parallel coordinates in three dimensions and sharp spectral isoperimetric inequalities}
\author{Anastasia V.~Vikulova}
\address{L.D.Landau Institute for Theoretical Physics, pr. Ak Semenova 1a, Chernogolovka, 142432, Russia}
\email{vikulovaav@gmail.com}
\date{}
\maketitle
	\begin{abstract}
	
In this paper we show how the method of parallel coordinates 
can be extended to three dimensions.
As an application, we prove the conjecture of 
Antunes, Freitas and Krej\v{c}i\v{r}\'ik \cite{AFK} 
that ``the ball maximises the first Robin eigenvalue 
with negative boundary parameter
among all convex domains of equal surface area''
under the weaker restriction that the boundary of the domain 
is diffeomorphic to the sphere and convex or axiconvex. 
We also provide partial results in arbitrary dimensions. 
	
	\vspace{11pt}
	
	
	\vspace{11pt}

	\end{abstract}
	
	\section{Introduction}

	The study of extremal eigenvalues of the Laplace operator is an active topic of research among mathematicians who are interested in spectral geometry. The most famous result in this field is certainly the Faber-Krahn inequality stating that the first Dirichlet eigenvalue of any domain is always greater than or equal to the first Dirichlet eigenvalue of the ball of the same volume (see~\cite{Faber_1923} and~\cite{Krahn_1924}). Similarly, the ball is an extremal domain for the first non-trivial Neumann eigenvalue, but now the ball is a maximiser (see~\cite{Szego_1954} and~\cite{Weinberger_1956}). For simply-connected planar domains, Weinstock proved that 
for the first Steklov eigenvalue the ball is a maximiser with the perimeter constraint~\cite{Weinstock_1954}. 

There also exists a Faber-Krahn-type inequality for the first eigenvalue of the Robin Laplacian with positive boundary parameter (see~\cite{Bossel_1986} and~\cite{Daners_2006}). On the other hand, Freitas and  Krej$\check{\textrm{c}}$i$\check{\textrm{r}}{\acute{\i}}$k established the surprising fact that the ball stops to be the extremal domain under the volume constraint for the Robin Laplacian with sufficiently large negative boundary parameter~\cite{FK7}. On the positive side, they proved the maximality of the disk among planar domains with fixed perimeter for the first eigenvalue of the Robin Laplacian with negative boundary parameter~\cite{AFK}. 
Similar optimisation results in the exterior of a compact set
were established in~\cite{KL1,KL2}.
The authors of~\cite{AFK} stated the conjecture 
that the reverse Faber-Krahn-type inequality
under the perimeter constraint
extends to higher dimensions under the convexity assumption~\cite[Conj.~4]{AFK}.
The primary motivation of this paper is to deal with this conjecture.
 
	Let $\Omega$ be a bounded domain in $\mathbb{R}^n$, $n \geq 1$,
with Lipschitz boundary~$\partial\Omega$. 
Let us consider the Robin eigenvalue problem	
\begin{equation}\label{problem}
	\left\{
	\begin{aligned}
	-\Delta u\,&=\,\lambda u \qquad  &&\mbox{in \,} \Omega \\
	\frac{\partial u}{\partial \nu}+\alpha u\,&=\,0 \qquad &&\mbox{on \,} 				\partial \Omega, \\
	\end{aligned}
	\right.
\end{equation}
where $\nu$ is an outer unit normal vector
and $\alpha$ is a real constant.
Notice that $\alpha=0$ coincides with the Neumann problem,
while $\alpha=\infty$ formally corresponds to the Dirichlet problem. 
We are interested in the first eigenvalue of~\eqref{problem},
denoted by $\lambda_1^{\alpha}(\Omega)$,
which admits the variational characterisation	
\begin{equation}\label{Rayleigh}
	\lambda_1^{\alpha}(\Omega)\,=\,\inf_{ 0 \ne u \in W^{1,2}(\Omega)}\frac{\|\nabla u\|^2_{L^2(\Omega)}+\alpha\|u\|^2_{L^2(\partial \Omega)}}{\|u\|^2_{L^2(\Omega)}}.
\end{equation}
We include the Dirichlet situation ``$\alpha=\infty$''
by the convention that then the boundary term in~\eqref{Rayleigh}
is suppressed and $W_0^{1,2}(\Omega)$ is taken for the test-function space.

We are mainly interested in the regime of negative~$\alpha$,
in which case $\lambda_1^{\alpha}(\Omega)$ is also negative.
Indeed, using a constant test function in~\eqref{Rayleigh}, 
we have 
 	
 	$$
 	\lambda_1^{\alpha}(\Omega)\,\leq\,\alpha\frac{\mathrm{Vol}_{n-1}(\partial \Omega)}{\mathrm{Vol}_n(\Omega)},
 	$$
where $\mathrm{Vol}_{k}(E)$ denotes the $k$-dimensional measure 
of a set $E \subset \mathbb{R}^n$.	
If $n=3$, we also write $\mathrm{Vol}_{2}(E)=:\mathrm{Area}(E)$.

We recall the following conjecture from~\cite{AFK}.
\begin{Conjecture}[{\cite[Conj.~4]{AFK}}]\label{Conj}
Let $\alpha\, \leq \, 0$ 
and $\Omega$ be a bounded convex domain in $\mathbb{R}^n$
with $n \geq 1$.
Then the $n$-ball maximises the first eigenvalue 
of problem~\eqref{problem} among all domains of
equal surface area.
\end{Conjecture}	
\begin{definition}
Let us recall that a Euclidean set is called \emph{axiconvex}
if it is rotationally invariant around an axis and its intersection
with any plane orthogonal to the symmetry axis is either a disk
or empty set (see~\cite{Dalphin}). 
\end{definition}

The following theorem establishes Conjecture~\ref{Conj}
in three dimensions.
\begin{Th}\label{Thm.main}
Let $\alpha\, \leq \, 0$, $\Omega$ be a bounded connected convex or axiconvex domain in $\mathbb{R}^3$, $\partial \Omega$ be a smooth boundary of the $\Omega$ and $\partial \Omega$ is diffeomorphic to the sphere. Then we have

$$
\lambda_{1}^{\alpha}(\Omega) \leq \lambda_{1}^{\alpha}(B), 
$$

\noindent where $B$ is the 3-ball and $\mathrm{Area}(\partial \Omega)\,=\,\mathrm{Area}(\partial B).$
\end{Th}	

For higher dimensions, we have the following partial result.
	\begin{Th}\label{Thm.high}
	Let $\alpha \leq 0$, $\Omega$ be a bounded connected domain in $\mathbb{R}^n$, $\partial \Omega$ be a smooth boundary of the $\Omega,$ $\bar \eta$ be a lower bound of the mean curvature of $\partial \Omega$ and assume that the following property holds true,

\begin{equation}\label{mean}
\bar{\eta}^{n-1} \geq  \frac{n \omega_n}{\mathrm{Vol}_{n-1}(\partial \Omega)},
\end{equation}

\noindent where $\omega_n$ is a volume of the unit n-ball, $\mathrm{Vol}_{n-1}(\partial \Omega)$ is a volume of $\partial \Omega$.

Then we have 

$$
\lambda_{1}^{\alpha}(\Omega) \leq \lambda_{1}^{\alpha}(B), 
$$

\noindent where $B$ is the n-ball and $\mathrm{Vol}_{n-1}(\partial \Omega)\,=\,\mathrm{Vol}_{n-1}(\partial B).$

	\end{Th}

\begin{Remark}
After writing this paper, we discovered a preprint by
Bucur, Ferone, Nitsch and Trombetti~\cite{Bucur-Ferone-Nitsch-Trombetti},
where Conjecture~\ref{Conj} is established in all dimensions.
In three dimensions, however, our Theorem~\ref{Thm.main} is still
of interest because it proves the conjecture
under the weaker assumption that $\partial \Omega$ is 
diffeomorphic to the sphere and convex or axiconvex instead of~$\Omega$ being merely convex
which is assumed in~\cite{Bucur-Ferone-Nitsch-Trombetti}.
Also our Theorem~\ref{Thm.high} covers some non-convex geometries.
Moreover, our method of proof differs in some aspects 
from that of~\cite{Bucur-Ferone-Nitsch-Trombetti} 
and makes it therefore of independent interest
for applications in other isoperimetric problems (see below). 
\end{Remark}

The proof of the Theorems~\ref{Thm.main} and~\ref{Thm.high}
is based on the method of parallel coordinates
as developed for the Dirichlet problem by Payne and Weinberger
\cite{Payne-Weinberger_1961} 
and modified for the Robin problem by Freitas and Krej\v{c}i\v{r}\'ik~\cite{FK7}.
The restriction to two dimensions in these papers seem to be crucial.
In fact, it is explicitly stated in~\cite{FK5} that
the proof of the results of~\cite{Payne-Weinberger_1961}
``does not seem to have a straightforward extension to higher dimensions''
(see~\cite[p.~8]{FK5}).
In this paper we manage to get this extension to three dimensions.
Indeed, the following result is a three-dimensional generalisation 
of \cite[Eq.~(2.15)]{Payne-Weinberger_1961}
as well as \cite[Thm.~4]{FK7}.
Moreover, we state it for all values of the boundary parameter~$\alpha$.

\begin{Th}\label{Thm.3D}
Let $\alpha \in \mathbb{R}\cup\{\infty\}$, 
$\Omega$ be a bounded connected  domain in $\mathbb{R}^3$, $\partial \Omega$ be a smooth boundary of the $\Omega$ and $\partial \Omega$ is diffeomorphic to the sphere and convex or axiconvex. Then we have
$$
\lambda^{\alpha}_1(\Omega)\, \leq \, \mu^{\alpha}_1(A_{R_1,R_2}),
$$
where $\mu^{\alpha}_1(A_{R_1,R_2})$ is the first eigenvalue of the Laplacian 
in the spherical shell  
$
  A_{R_1,R_2} := \{x \in \mathbb{R}^3: R_1<|x|<R_2\}
$
with radii
\begin{equation}\label{radii}
\begin{aligned}
   R_1\,&=\,\frac{\sqrt[3]{\mathrm{Area}^{\frac{3}{2}}(\partial \Omega)-6\pi^{\frac{1}{2}}\mathrm{Vol}(\Omega)}}{2\pi^{\frac{1}{2}}},
\\
R_2\,&=\,\frac{\mathrm{Area}^{\frac{1}{2}}(\partial \Omega)}{2\pi^{\frac{1}{2}}},
\end{aligned}
\end{equation}
subject to the Neumann boundary condition at the inner boundary of~$A_{R_1,R_2}$
and the Robin boundary condition~\eqref{problem} at the outer boundary.  
\end{Th}

We also have a similar result in higher dimensions
under the hypothesis~\eqref{mean},
see Proposition~\ref{Prop.high}.

	\section{Preliminaries}
	
\subsection{Parallel coordinates in any dimension}
	Let $M$ be an $n$-dimensional Riemannian manifold with metric $g$. Let $\Omega$ be a domain with smooth boundary in $M$, and let $\rho: \, \Omega \to \mathbb{R}$ be the distance function from the boundary and $\rho(x)\,=\,\mathrm{dist}(x,\partial \Omega).$ Define the \emph{normal exponential map} $\Phi: \, [0,\infty) \times \partial \Omega \to M$ by setting	
	$$
	\Phi(t,\xi)\,=\,\exp_{\xi}t\nu(\xi),
	$$
where $\nu(\xi)$ is the unit normal vector to $\partial \Omega$
that we choose to be oriented inside $\Omega$. Define the \emph{cut-radius} map $c: \, \partial \Omega \to \mathbb{R}$ by the following property: the geodesic $\Phi_{\xi}\,=\,\Phi(\cdot,\xi):\,[0,t] \to M$ minimises the distance from $\partial \Omega$ 
if and only if $t \in [0,c(\xi)]$. 

The cut-radius map $c$ is known to be continuous and we clearly have $\max \, c \,=\,R$, where $R$ is the radius of the largest inscribed ball in~$\Omega$, 
which is called the inner radius of $\Omega$. The set of all points $\Phi(c(\xi),\xi)$, as $\xi$ runs in $\partial \Omega$, is called the \emph{cut-locus} of $\Omega$, and is denoted by 

$\mathrm{Cut}(\Omega).$ It is a closed subset of $\Omega,$ and it has zero measure. The map $\Phi$, when restricted to the open set $U\,=\,\{(t,\xi) \in (0,\infty) \times \partial \Omega:\, 0 \, < \, t \, < \, c(\xi)\}$, is a diffeomorphism onto $\Phi(U)\,=\,\Omega \setminus \mathrm{Cut}(\Omega)$. 

The pair $(t,\xi)$ is called \emph{parallel coordinates} based at $\partial \Omega.$ Let $dv_n$ be the Riemannian volume form of $\Omega$. Pulling it back by $\Phi$, we can write $\Phi^{*}(dv_n)(t,\xi)\,=\,h(t,\xi)dt\, dv_{n-1},$ where $dv_{n-1}(\xi)$ is the volume form of $\partial \Omega$ for the induced metric. As $\Phi$ is smooth, its Jacobian $h$ will also be smooth and positive on $U$. 

The following inequality for the Jacobian has been given by Heintze and Karcher (see the papers~\cite{Heintze-Karcher_1978} and~\cite{Savo_2001}):
$$
h(t,\xi)\, \leq \, [s'_{\bar K}(t)-\bar \eta s_{\bar K}(t)]^{n-1}_{+},
$$
where $a_+$ denotes $\max\{a,0\}$, $\bar K$ is a lower bound of the Ricci curvature of $\Omega$, $\bar \eta$ is a lower bound of the mean curvature of $\partial \Omega$ and 
$$
s_{\bar K}(t)\,=\,\begin{cases}
\frac{1}{\sqrt{\bar K}}\sin(t\sqrt{\bar K}),&\text{if $\bar K\,>\,0$;}\\
t,&\text{if $\bar K\,=\,0$;}\\
\frac{1}{\sqrt{|\bar K|}}\sinh(t\sqrt{\left|\bar K\right|}),&\text{if $\bar K\,<\,0$.}
\end{cases} 
$$
In our case, where $\Omega$ is a domain in $\mathbb{R}^n$ 
and the metric $g$ is an Euclidean metric, we have 
\begin{equation}\label{crude}
h(t,\xi)\, \leq \, [1-\bar \eta t]^{n-1}_{+}.
\end{equation}

Denote $\Omega_t = \{x \in \Omega:\,0\,<\,\rho(x)\,<\,t\}$.
Expressing $\mathrm{Vol}_{n-1}(\partial \Omega_t)$ in parallel coordinates, we obtain 
$$
\begin{aligned}
\mathrm{Vol}_{n-1}(\partial \Omega_t)\,
&=\,\int_{\{\xi \in \partial \Omega: c(\xi)>t\}} h(t,\xi)\, %
\\
&\leq \, [1-\bar \eta t]^{n-1}_{+} \mathrm{Vol}_{n-1}(\partial \Omega). 
\end{aligned}
$$
We say that  $\mathrm{Vol}_n(\Omega_t)$ is the volume of the shell $\Omega_t$. 
So we have $|\mathrm{Vol}_n(\Omega_t)-\mathrm{Vol}_n(\Omega_s)|\,=\,|\int^t_s \mathrm{Vol}_{n-1}(\partial \Omega_z) dz|$, and as the integrand admits a uniform upper bound, we see that $\mathrm{Vol}_n(\Omega_t)$ is uniformly Lipschitz on $[0,R]$ and we get $\mathrm{Vol}_n'(\Omega_t)\,=\,\mathrm{Vol}_{n-1}(\partial \Omega_t)$.
By definition, we put $\mathrm{Vol}_{n-1}(\partial \Omega_0)\,:=\,\mathrm{Vol}_{n-1}(\partial \Omega)$ and $\mathrm{Vol}_n(\Omega_R)\,:=\,\mathrm{Vol}_n(\Omega).$

\subsection{An inequality between eigenvalues in balls and spherical shells}
Let $B_{R}$ be an $n$-dimensional ball of radius $R$. By the rotational symmetry and regularity, we have
$$
\lambda^{\alpha}_1(B_R)\,=\,\inf_{0 \ne \phi \in C^{\infty}([0,R])} \frac{\int^R_0 |\phi'(r)|^2 r^{n-1} dr\,+\,\alpha R^{n-1} |\phi(R)|^2}{\int^R_0 |\phi(r)|^2 r^{n-1} dr}.
$$	
We include the Dirichlet situation ``$\alpha=\infty$''
by the convention that then the second term in the numerator
is suppressed and the functions~$\phi$ are additionally
assumed to to satisfy $\phi(R)=0$.

Now, let $A_{R_1,R_2} = B_{R_2}\setminus\overline{B_{R_1}}$ 
be an $n$-dimensional spherical shell 
and $R_1\,<\,R_2$ are the radii of the outer and inner spheres. 
Let $\mu^{\alpha}_1(A_{R_1,R_2})$ be the first eigenvalue of the Laplacian with Robin boundary condition with $\alpha$ on the outer sphere and with Neumann boundary condition on the inner sphere. We have
$$
\mu^{\alpha}_1(A_{R_1,R_2})\,=\,\inf_{0 \ne \phi \in C^{\infty}([R_1,R_2])} \frac{\int^{R_2}_{R_1} |\phi'(r)|^2 r^{n-1} dr\,+\,\alpha R^{n-1}_2 |\phi(R_2)|^2}{\int^{R_2}_{R_1} |\phi(r)|^2 r^{n-1} dr}.
$$
As above, we include the Dirichlet situation ``$\alpha=\infty$''
by the convention that then the second term in the numerator
is suppressed and the functions~$\phi$ are additionally
assumed to to satisfy $\phi(R_2)=0$.

The following Propisition is an extension of \cite[Prop.~5]{AFK}
to any dimension.

\begin{Prop}\label{Prop.comparison}
	Let $\alpha \, \leq \, 0.$ For any $0\,<\,R_1\, < \, R_2,$ we have 
	$$
	\mu^{\alpha}_1(A_{R_1,R_2})\, \leq \lambda^{\alpha}_1(B_{R_2}).
	$$
	\end{Prop}

\begin{proof}
By symmetry, $\lambda_1^{\alpha}(B_{R_2})$ 
is the smallest solution of the problem
$$
\left\{
\begin{aligned}
-r^{-(n-1)}[r^{n-1}\phi'(r)]'\,&=\,\lambda \phi(r)\\
\phi'(0)\,&=\,0\\
\phi'(R_2)+\alpha \phi(R_2)\,&=\,0\\
\end{aligned}
\right.
$$
where $r \in [0,R_2].$
We obviously have 
$$
[r^{n-1}\phi(r)\phi'(r)]'\,=\,-\lambda_1^{\alpha}(B_{R_2})r^{n-1}\phi(r)^2+r^{n-1}\phi'(r)^2\,\geq\,0.
$$
So $r^{n-1}\phi(r)\phi'(r)$ is non-decreasing.

Let $\phi$ be the eigenfunction for this problem and $\phi$ normalises to one in $L^2((0,R_2),r^{n-1} dr).$ In these conditions and with the integrating by parts, we get
$$
\mu_1^{\alpha}(A_{R_1,R_2})\, \leq \, \lambda_1^{\alpha}(B_{R_2})-R_1^{n-1} \phi(R_1)\phi'(R_1). 
$$
 So we have proved that $\mu_1^{\alpha}(A_{R_1,R_2})\, \leq \, \lambda_1^{\alpha}(B_{R_2}).$
\end{proof}

\section{Proof of Theorems~\ref{Thm.main} and~\ref{Thm.3D}}
 	
 Let us pick a smooth function $\phi:\,[0,\mathrm{Vol}(\Omega)] \to \mathbb{R}$  and test-function $u\,=\,\phi \circ \mathrm{Vol} \circ \rho$ which is Lipschitz in $\Omega$. In our case, the Jacobian is equal to $h(\xi,t)\,=\,1-2M(\xi)t+K(\xi)t^2,$ where $M(\xi)$ is the mean curvature and $K(\xi)$ is the Gaussian curvature of $\partial\Omega$. 
So employing the parallel coordinates together with the co-area formula
(see~\cite{Savo_2001}), we obtain
$$
\begin{aligned}
\|u\|^2_{L^2(\Omega)}\,&=\,\int^R_0\phi^2(\mathrm{Vol}(\Omega_t))\mathrm{Vol}'(\Omega_t)dt;
\\
\|\nabla u\|^2_{L^2(\Omega)}\,&=\,\int^R_0 (\phi'(\mathrm{Vol}(\Omega_t)))^2 (\mathrm{Vol}'(\Omega_t))^3dt;
\\
\|u\|^2_{L^2_{(\partial \Omega)}} &=\, \mathrm{Area}(\partial \Omega) \phi^2(0).
\end{aligned}
$$

Now we make the change of variables

$$
r(t):=\frac{\sqrt[3]{\mathrm{Area}^{\frac{3}{2}}(\partial \Omega)-6\pi^{\frac{1}{2}}\mathrm{Vol}(\Omega_t)}}{2\pi^{\frac{1}{2}}}.
$$
Defining 
$\psi(r)\,=\,\phi \left( \frac{\mathrm{Area}^{3/2}(\partial \Omega)}{6 \pi^{1/2}}-\frac{4}{3}\pi r^3(t) \right)$, it follows that
$$
\begin{aligned}
\|u\|^2_{L^2(\Omega)}\,&=\,\int^{R_2}_{R_1} 4 \pi r^2 \psi^2(r)dr;
\\
\|\nabla u\|^2_{L^2(\Omega)}\,&=\,\int^{R_2}_{R_1} (\psi'(r))^2 4 \pi r^2 (r')^2 dr;
\\
\|u\|^2_{L^2(\partial \Omega)} &=\, \mathrm{Area}(\Omega) \psi^2(r_2),
\end{aligned}
$$

where $R_2=r(0)$ and $R_1=r(R)$, see~\eqref{radii}.
So we get	
$$
\begin{aligned}
\lambda^{\alpha}_1(\Omega) &\le \inf_{\psi \ne 0} \frac{\int^{R_2}_{R_1} (\psi'(r))^2 4 \pi r^2 (r')^2 dr+\alpha \mathrm{Area}(\partial \Omega) \psi^2(R_2)}{\int^{R_2}_{R_1} 4 \pi r^2 \psi^2(r)dr}
\\ 
&= 
\inf_{\psi \ne 0} \frac{\int^{R_2}_{R_1} (\psi'(r))^2  r^2 (r')^2 dr+\alpha R^2_2 \psi^2(R_2)}{\int^{R_2}_{R_1} r^2 \psi^2(r)dr}.
\end{aligned}
$$

The following proposition concludes the proof of Theorem~\ref{Thm.3D}.
\begin{Prop}\label{Prop.prime}
One has $|r'(t)|\, \leq \, 1$ for all $t \in [0,R]$.
\end{Prop}
\begin{proof}
First, using the definition of $r(t)$, we obtain
$$
r'(t)\,=\,-\frac{\mathrm{Area}(\partial \Omega_t)}{(\mathrm{Area}(\Omega)^{3/2}-6 \pi^{1/2} \mathrm{Vol}(\Omega_t))^{2/3}}.
$$
For the area we get
$$
\begin{aligned}
\mathrm{Area}(\partial \Omega_t)\,&=\, \int_{\{\xi \in \partial \Omega, t<c(\xi),\Phi(\xi,t) \in \Omega\}} (1-2M(\xi)t+K(\xi)t^2) ds 
\\
 &\le \mathrm{Area}(\partial \Omega)-2tm+t^2 2 \pi \chi (\partial \Omega),
\end{aligned}
$$
where $m\,=\,\int_{\partial \Omega}M(\xi)$
and $\chi (\partial \Omega)$ is the Euler characteristics of $\partial\Omega$.

For the latter we have $\chi (\partial \Omega)=2$
provided that $\partial\Omega$ is diffeomorphic to the sphere.
For the former we use the inequality (see~\cite[Chapt.~4]{BuZa} and~\cite{Dalphin})
$$
m\,\geq \, 2 \pi^{\frac{1}{2}}\mathrm{Area}(\partial \Omega)^{\frac{1}{2}}.
$$
Consequently,
$$
\begin{aligned}
\mathrm{Area}(\partial \Omega_t)\, &\leq \, \mathrm{Area}(\partial \Omega)-2tm+t^2 4\pi,
\\
\mathrm{Vol}(\Omega_t)\, &\leq \, t \mathrm{Area}(\partial \Omega)-t^2 m+t^3 \frac{4}{3}\pi.
\end{aligned}
$$

By calculation we obtain
$$
\begin{aligned}
\lefteqn{
\mathrm{Area}^{3/2}(\partial \Omega_t)
}
\\ 
&\le (\mathrm{Area}(\partial \Omega)-2tm+t^2 4 \pi)^{3/2} 
\\
&= (\mathrm{Area}(\partial \Omega)-2tm+t^2 4 \pi)
(\mathrm{Area}(\partial  \Omega)-2tm+t^2 4 \pi)^{1/2}
\\
&
\le \, (\mathrm{Area}(\partial \Omega)-2tm+t^2 4 \pi)(\mathrm{Area}(\partial \Omega)-2t(2 \pi^{1/2} \mathrm{Area}^{\frac{1}{2}}(\partial \Omega))+t^2 4 \pi)^{1/2} 
\\
&
=\, (\mathrm{Area}(\partial \Omega)-2tm+t^2 4 \pi)(\mathrm{Area}^{\frac{1}{2}}(\partial \Omega)-2 \pi^{\frac{1}{2}} t)
\\
&
= \mathrm{Area}^{\frac{3}{2}}(\partial \Omega)-\mathrm{Area}^{\frac{1}{2}}(\partial \Omega) 2tm+4 \pi t^2 \mathrm{Area}^{\frac{1}{2}}(\partial \Omega)
\\
& \qquad
-2 \pi^{\frac{1}{2}} t \mathrm{Area}(\partial \Omega)+4 \pi^{1/2} m t^2-8 \pi^{\frac{3}{2}}t^3.
\end{aligned}
$$
So it remains to prove that
\begin{multline*}
2\mathrm{Area}^{\frac{1}{2}}(\partial \Omega) t m-4 \pi t^2 \mathrm{Area}^{\frac{1}{2}}(\partial \Omega)+2 \pi^{1/2} t \mathrm{Area}(\partial \Omega)-4 \pi^{1/2}mt^2+8\pi^{\frac{3}{2}}t^3 
\\
\ge \, 6\pi^{1/2}t \mathrm{Area}(\partial \Omega)-6\pi^{1/2}t^2m+8\pi\pi^{1/2}t^3.
\end{multline*}
Equivalently, 
transferring everything to the left part,
$$
t^2(2 \pi^{1/2}m-4\pi \mathrm{Area}^{\frac{1}{2}}(\partial \Omega))+t(2 \mathrm{Area}^{\frac{1}{2}}(\partial \Omega) m-4\pi^{1/2}\mathrm{Area}(\partial \Omega))\geq\,0.
$$
But the left-hand side is evidently 
greater than or equal to zero. 
It concludes the proof of the propisition.
\end{proof}

Combining Theorem~\ref{Thm.3D} with Proposition~\ref{Prop.comparison}, 
we get Theorem~\ref{Thm.main}.

\section{Proof of the Theorem~\ref{Thm.high}}

First of all, let us establish the following extension
of Theorem~\ref{Thm.3D} to higher dimensions. 

\begin{Prop}\label{Prop.high}
Let $\alpha \in \mathbb{R}\cup\{\infty\}$
and $\Omega$ be a smooth bounded connected domain in $\mathbb{R}^n$
with any $n \geq 1$.
Assume~\eqref{mean}. 
Then
$$
\lambda^{\alpha}_1(\Omega)\, \leq \, \mu^{\alpha}_1(A_{R_1,R_2}),
$$
where
$$
\begin{aligned}
R_1\,&=\,\left(\frac{\Gamma(\frac{n}{2})}{2\pi^{\frac{n}{2}}}\right)^{\frac{1}{n-1}}\left(\mathrm{Vol}_{n-1}^{\frac{n}{n-1}}(\partial \Omega)-\mathrm{Vol}_n(\Omega)n^{\frac{n}{n-1}}\omega_n^{\frac{1}{n-1}} \right)^{\frac{1}{n}},
\\
R_2\,&=\,\frac{\Gamma(\frac{n}{2})}{2\pi^{\frac{n}{2}}}^{\frac{1}{n-1}}\mathrm{Vol}_{n-1}(\partial \Omega)^{\frac{1}{n-1}},
\end{aligned}
$$
where $\Gamma$ is the Gamma function.
\end{Prop}
\begin{proof}
We will prove this proposition as above for $\mathbb{R}^3$. We take a smooth function $\phi:\,[0,\mathrm{Vol}_{n}(\Omega)] \to \mathbb{R}$  and a test-function $u\,=\,\phi \circ \mathrm{Vol}_{n} \circ \rho$ which is Lipschitz in $\Omega.$ Employing the parallel coordinates together with the co-area formula (see~\cite{Savo_2001}), we obtain
$$
\begin{aligned}
\|u\|^2_{L^2(\Omega)}\,&=\,\int^R_0\phi^2(\mathrm{Vol}_n(\Omega_t))\mathrm{Vol}_n'(\Omega_t)dt;
\\
\|\nabla u\|^2_{L^2(\Omega)}\,&=\,\int^R_0 (\phi'(\mathrm{Vol}_n(\Omega_t)))^2 (\mathrm{Vol}_n'(\Omega_t))^3 dt;
\\
\|u\|^2_{L^2 (\partial \Omega)} \,&=\, \mathrm{Vol}_{n-1}(\partial \Omega) \phi^2(0).
\end{aligned}
$$

Now we need to change the variables
$$
r(t):=\left(\frac{\Gamma(\frac{n}{2})}{2\pi^{\frac{n}{2}}}\right)^{\frac{1}{n-1}}\left(\mathrm{Vol}_{n-1}^{\frac{n}{n-1}}(\partial \Omega)-\mathrm{Vol}_n(\Omega_t)n^{\frac{n}{n-1}}\omega_n^{\frac{1}{n-1}} \right)^{\frac{1}{n}}.
$$
Defining $\psi(r):=\phi\left( \mathrm{Vol}_{n-1}^{\frac{n}{n-1}}(\partial \Omega)n^{\frac{-n}{n-1}}\omega_n^{\frac{-1}{n-1}}-r^n\frac{2\pi^{\frac{n}{2}}}{\Gamma(\frac{n}{2})}n^{\frac{-n}{n-1}}\omega_n^{\frac{-1}{n-1}} \right)$, it follows that
$$
\begin{aligned}
\|u\|^2_{L^2(\Omega)}\,&=\,\int^{R_2}_{R_1}\psi^2(r)r^{n-1}\frac{2\pi^{\frac{n}{2}}}{\Gamma(\frac{n}{2})}dr;
\\
\|\nabla u\|^2_{L^2(\Omega)}\,&=\,\int^{R_2}_{R_1}(\psi'(r))^2\frac{r^{n-1}2\pi^{\frac{n}{2}}}{\Gamma(\frac{n}{2})}(r'(t))^2dr;
\\
\|u\|^2_{L^2(\partial \Omega)}\,&=\, \mathrm{Vol}_{n-1}(\partial \Omega)\psi^2(R_2).
\end{aligned}
$$
So the change of the variables gives
$$
\begin{aligned}
\lambda_1^{\alpha}(\Omega)\,
&\leq\,\inf_{0 \ne \psi \in C^{\infty}([R_1,R_2])}\frac{\int^{R_2}_{R_1}(\psi'(r))^2\frac{r^{n-1}2\pi^{\frac{n}{2}}}{\Gamma(\frac{n}{2})}(r'(t))^2dr+\alpha \mathrm{Vol}_{n-1}(\partial \Omega)\psi^2(R_2)}{\int^{R_2}_{R_1}\psi^2(r)r^{n-1}\frac{2\pi^{\frac{n}{2}}}{\Gamma(\frac{n}{2})}dr} 	
\\
&=\,\inf_{0 \ne \psi \in C^{\infty}([R_1,R_2])}\frac{\int^{R_2}_{R_1}\psi'(r)^2 r^{n-1}(r'(t))^2dr+\alpha\psi^2(R_2)R_2^{n-1}}{\int^{R_2}_{R_1}\psi^2(r)r^{n-1}dr}.
\end{aligned}
$$

It remains to prove that $|r'(t)|\, \leq \, 1.$
First, using the definition of $r(t)$, we obtain
$$
r'(t)\,=\,\frac{-V_{n-1}(\partial \Omega_t)}{\left( V_{n-1}^{\frac{n}{n-1}}(\partial \Omega)-\mathrm{Vol}_n(\Omega_t)n^{\frac{n}{n-1}}\omega_{n}^{\frac{1}{n-1}}\right)^{\frac{n-1}{n}}}.
$$
We claim that
$$
V_{n-1}^{\frac{n}{n-1}}(\partial \Omega_t)\,\leq \, V_{n-1}^{\frac{n}{n-1}}(\partial \Omega)-\mathrm{Vol}_n(\Omega_t)n^{\frac{n}{n-1}}\omega_{n}^{\frac{1}{n-1}}.
$$
Indeed, this follows from
$$
\mathrm{Vol}_{n-1}(\partial \Omega_t)\, \leq \, (1-\bar \eta t)_{+}^{n-1} \mathrm{Vol}_{n-1}(\partial \Omega)
$$
and
$$
\mathrm{Vol}_{n}(\Omega_t)\, \leq \, \frac{1}{n \bar \eta}(1-(1-\bar \eta t)^n)_{+} \mathrm{Vol}_{n-1}(\partial \Omega),
$$
where we have used~\eqref{crude}. 
Then we have
$$
\begin{aligned}
\lefteqn{
\mathrm{Vol}_{n-1}^{\frac{n}{n-1}}(\partial \Omega)-\mathrm{Vol}_n(\Omega_t)n^{\frac{n}{n-1}}\omega_n^{\frac{1}{n-1}}-\mathrm{Vol}_{n-1}^{\frac{n}{n-1}}(\partial \Omega_t)
}
\\
&\geq
\mathrm{Vol}_{n-1}^{\frac{n}{n-1}}(\partial \Omega)-n^{\frac{1}{n-1}} \omega^{\frac{1}{n-1}}\frac{1}{\bar \eta}(1-(1-\bar \eta t)^n)_{+}\mathrm{Vol}_{n-1}(\partial \Omega)-(1-\bar \eta t)^n_{+} \mathrm{Vol}^{\frac{n}{n-1}}_{n-1}(\partial \Omega)
\\
& \ge
\mathrm{Vol}_{n-1}(\partial \Omega)\left(\mathrm{Vol}_{n-1}^{\frac{1}{n-1}}(\partial \Omega)(1-(1-\bar \eta t)^n_{+})-n^{\frac{1}{n-1}} \omega^{\frac{1}{n-1}}\frac{1}{\bar \eta}(1-(1-\bar \eta t)^n)_{+}\right)
\\
&\geq \, 0.
\end{aligned}
$$
Here the last inequality holds true due to~\eqref{mean}.
This implies $|r'(t)|\,\leq\,1$ and therefore
yields that $\lambda_1^{\alpha}(\Omega)\,\leq\,\mu_1^{\alpha}(A_{R_1,R_2}).$
\end{proof}

For $R_1$ and $R_2$ as in Proposition~\ref{Prop.high} 
we have $\mathrm{Vol}_{n}(A_{R_1,R_2})\,=\,\mathrm{Vol}_{n}(\Omega)$ and $\mathrm{Vol}_{n-1}(\partial B_{R_2})\,=\,\mathrm{Vol}_{n-1}(\partial \Omega).$

Combining Proposition~\ref{Prop.high} 
with Proposition~\ref{Prop.comparison}, 
we obtain $\lambda_1^{\alpha}(\Omega)\,\leq\,\mu_1^{\alpha}(A_{R_1,R_2})\,\leq\,\lambda_1^{\alpha}(B_{R_2}).$ This completes the proof of Theorem~\ref{Thm.high}.

\begin{Remark}
Let us give examples of domains for which the inequality~\eqref{mean} holds.

First, let us look at the ellipsoid of revolution. 
The ellipsoid may be parameterized as 

$$
\left\{
\begin{aligned}
z\,=\,a \cos(\theta) \cos(\phi),\\
y\,=\,a \cos(\theta) \sin(\phi),\\
x\,=\,c \sin(\theta).\\
\end{aligned}
\right.
$$

Suppose $c\,>\,a;$ then $a\,=\,mc,$ where $m \in (0,1).$
The lower bound of the mean curvature is equal to $\bar \eta\,=\,\frac{c^2+a^2}{2ac^2}.$ The area of this ellipsoid is equal to $\mathrm{Vol}_{n-1}(\partial \Omega)\,=\, 2\pi a \left(a+\frac{c^2}{\sqrt{c^2-a^2}}\arcsin(\frac{\sqrt{c^2-a^2}}{c}) \right).$
So, the inequality is true when
$$
m(1+m^2)^2\left(m+\frac{\arcsin(\sqrt{1-m^2})}{\sqrt{1-m^2}}\right) \geq 8m^2.
$$
Obviously, such $m$ exist in the interval $ (0,1).$
	
Secondly, let us look at the torus. The torus may be parameterized as 

$$
\left\{
\begin{aligned}
x\,=\,(R+r\cos(\theta))\cos(\phi),\\
y\,=\,(R+r\cos(\theta))\sin(\phi),\\
z\,=\,r\sin(\theta).\\
\end{aligned}
\right.
$$
The lower bound of the mean curvature is equal to $\bar \eta\,=\,\frac{R-2r}{r(R-r)}.$ Suppose $R\,>\,r.$ Let us $r\,=\,mR,$ where $m \in (0,\frac{1}{2}).$  The area of this ellipsoid is equal to $\mathrm{Vol}_{n-1}(\partial \Omega)\,=\, 4 \pi^2 Rr.$
So, the inequality is true when
$$
(1-2m)^2 \pi \geq (1-m)^2 m.
$$
Obviously, such $m$ exist in the interval $ (0,\frac{1}{2}).$	
\end{Remark}

\bigskip
\subsection*{Acknowledgment}
%

This work is supported by the Russian Science Foundation under grant 18-11-00316 Geometric methods in non-linear problems of mathematical physics.
 
The author is indebted to A.~V.~Penskoi who ispired the interest to spectral geometry for suggesting this problem. The author also is indebted to O.~I.~Mokhov for fruitful discussions and important remarks. The author is grateful to D.~Krej\v{c}i\v{r}\'ik
and V.~Lotoreichik for useful discussions and to the Czech Technical University for its hospitality.

\bigskip


\subsection*{Conflict of interest}
%

There is no conflict of interest.

%
%

\providecommand{\bysame}{\leavevmode\hbox to3em{\hrulefill}\thinspace}
\providecommand{\MR}{\relax\ifhmode\unskip\space\fi MR }
\providecommand{\MRhref}[2]{%
  \href{http://www.ams.org/mathscinet-getitem?mr=#1}{#2}
}
\providecommand{\href}[2]{#2}

\end{document}